\documentclass[11pt]{amsart}
\usepackage{amsmath,amsfonts,amsthm,amssymb,amscd,hyperref}
\def\classification#1{\def\@class{#1}}
\classification{\null}

\DeclareFontFamily{OT1}{rsfs}{}
\DeclareFontShape{OT1}{rsfs}{n}{it}{<-> rsfs10}{}
\DeclareMathAlphabet{\mathscr}{OT1}{rsfs}{n}{it}

\DeclareMathOperator{\supp}{supp}

\DeclareMathOperator{\diam}{diam}
\DeclareMathOperator{\leng}{length}

\DeclareMathOperator{\mix}{mix}

\DeclareMathOperator{\Sym}{Sym}
\DeclareMathOperator{\Prob}{Prob}
\DeclareMathOperator{\Alt}{Alt}

\DeclareMathOperator{\tr}{tr}

\DeclareMathOperator{\Id}{Id}

\newtheorem{prop}{Proposition}[section]
\newtheorem{thm}[prop]{Theorem}

\newtheorem{lem}[prop]{Lemma}

\theoremstyle{remark}

\numberwithin{equation}{section}

\begin{document}

\title[Random generators of the symmetric group]{Random generators of the symmetric
group:\\ diameter, mixing time and spectral gap}
\author{Harald A. Helfgott}
\address{Harald A. Helfgott, 
\'Ecole Normale Sup\'erieure, D\'epartement de Math\'ematiques, 45 rue d'Ulm, F-75230 Paris, France}
\email{harald.helfgott@ens.fr}
\author{\'{A}kos Seress}
\address{\'{A}kos Seress (deceased),
Centre for the Mathematics of
Symmetry and Computation,
The University of Western Australia,
Crawley, WA 6009 Australia,
and
Department of Mathematics,
The Ohio State University,
Columbus, OH 43210, USA}
\email{akos@math.ohio-state.edu}
\author{Andrzej Zuk}
\address{Andrzej Zuk,
Institut de Math\'ematiques,
Universit\'e Paris 7,
13 rue Albert Einstein,
75013 PARIS}
\email{zuk@math.jussieu.fr}

\begin{abstract}
Let $g$, $h$ be a random pair of generators of $G=\Sym(n)$ or $G=\Alt(n)$.
We show that, with probability tending to $1$ as $n\to \infty$, (a)
the diameter of $G$ with respect to $S = \{g,h,g^{-1},h^{-1}\}$ is at 
most $O(n^2 (\log n)^c)$, and (b) the mixing time of 
$G$ with respect to $S$
is at most $O(n^3 (\log n)^c)$. (Both $c$ and the implied constants are absolute.)

These bounds are far lower than the strongest worst-case bounds known
(in Helfgott--Seress, 2013); they roughly match 
the worst known examples. We also give an improved, though still
non-constant, bound on the spectral gap.

Our results rest on a combination of the algorithm in (Babai--Beals--Seress,
2004) and the fact that the action of a pair of random permutations
is almost certain to act as an expander on $\ell$-tuples, where $\ell$
is an arbitrary constant (Friedman et al., 1998).
\end{abstract}

\maketitle

\section{Introduction}
\subsection{Results}
Let $G$ be a finite group. 
Let $S$ be a set of generators of $G$; assume $S = S^{-1}$. The (undirected)
{\em Cayley graph} $\Gamma(G,S)$ is the graph having the elements of $G$
as its vertices and the pairs $\{g,g s\}$ ($g\in G$, $s\in S$) as its
edges. The {\em diameter} of $G$ with respect to $S$ is the diameter
$\diam(\Gamma(G,S))$ of the Cayley graph $\Gamma(G,S)$:
\[\diam(\Gamma(G,S)) = \max_{g_1, g_2\in G} \mathop{\min_{\text{$P$ a
      path}}}_{\text{from $g_1$ to $g_2$}} \leng(P),\]
where the {\em length} of a path is the number of edges it traverses.
In other words, $\diam(\Gamma(G,S))$ is the maximum, for $g\in G$, of
the length $\ell$ of the shortest expression $g = s_1 s_2 \dotsc s_\ell$
with $s_i\in S$.

\begin{thm}\label{thm:diam}
Let $S = \{g,h,g^{-1},h^{-1}\}$, where $g$, $h$ are elements of
$\Sym(n)$ taken at random, uniformly and independently.
Let $G = \langle S\rangle$. Then, with probability
$1-o(1)$, the diameter $\diam(\Gamma(G,S))$ of $G$ with respect to $S$
is at most $O(n^2 (\log n)^c)$, where $c$ and the implied
constant are absolute.
\end{thm}

In the study of permutation groups, bounds are wanted not just for the
diameter but also for two closely related quantities that give a 
finer description of the quality of a generating set $S$.
The {\em spectral gap} is the difference $\lambda_0-\lambda_1$ between the
two largest eigenvalues $\lambda_0$, $\lambda_1$ (where $\lambda_0=1$ and
$\lambda_0\geq \lambda_1$)
of the normalized adjacency matrix $\mathscr{A}$ on $\Gamma(G,S)$, seen
as an operator on functions $f:G\to \mathbb{C}$:
\begin{equation}\label{eq:lublu}\mathscr{A} f(g) := \frac{1}{|S|} \sum_{h\in S} f(g h) .\end{equation}

The other quantity is the {\em mixing time}. A lazy random walk on $\Gamma(G,S)$
consists of taking $x_1,x_2,\dotsc \in G$ at random and independently with
distribution \begin{equation}\label{eq:ocio}
\mu = \frac{1}{2} 1_{\{e\}} + \frac{1}{2 |S|} 1_S,\end{equation}
where $1_A(x) = 1$ if $x\in A$ and $1_A(x)=0$ if $x\notin A$; the outcome of
the lazy random walk of length $k$ is $x_1 x_2 \dotsb x_k$. The 
{\em $(\epsilon,d)$-mixing time} $t_{\mix,\epsilon,d}$ is the least $k$ such that the distribution
$\mu^{(k)} = \mu \ast \mu \ast \dotsb \ast \mu$ of the outcome of a lazy random
walk of length $k$ is very close to the uniform distribution $1_G/|G|$ on $G$:
\[d\left(\mu^{(k)},\frac{1}{|G|} 1_G\right) \leq \epsilon,\]
where $\epsilon>0$ and $d$ is a distance function on $\mathbb{C}^G$ (e.g.,
$d=\ell^1, \ell^2, \ell^\infty$). What may be called the {\em strong}
($\epsilon$-)mixing time corresponds to the distance function $|G|\cdot
\ell^\infty$, i.e., the $\ell^\infty$ norm scaled by a factor of $|G|$;
in other words, the strong mixing time with respect to $\epsilon$ equals 
$t_{\mix,\epsilon/|G|,\ell^\infty}$. The mixing time is defined in \cite{MR2466937}
(and several other sources)
with respect to the {\em total variation distance} $\text{TV}$, which is simply 
$(|G|/2) \ell^1$ (where $\ell^1$ is scaled so that $\ell^1(1_G)=1$,
and where $|G|$, as usual, denotes the number of elements of $G$). Thus,
 the mixing time as in \cite{MR2466937} (which sets $\epsilon=1/4$;
see \cite[(4.33)]{MR2466937}) equals $t_{\mix,1/(2|G|),\ell^1}$, which is 
bounded by the strong $(1/2)$-mixing time, i.e.,
$t_{\mix,1/(2|G|),\ell^\infty}$ (since $|\cdot|_1 \leq |\cdot|_\infty$ on a space
of measure $1$). It is easy to check that the strong $(1/2)^k$-mixing time
is bounded by $k$ times the strong $(1/2)$-mixing time. This motivates us to
define {\em strong mixing time} to mean the strong $(1/2)$-mixing time;
thus, as we were saying, the mixing time (defined as usual, with respect to the TV norm and $\epsilon=1/4$) is bounded by the strong mixing time.

\begin{thm}\label{thm:mix}
Let $S = \{g,h,g^{-1},h^{-1}\}$, where $g$, $h$ 
are elements of
$\Sym(n)$ taken at random, uniformly and independently.
Let $G = \langle S\rangle$. Then, with probability
$1-o(1)$, the spectral gap is bounded from below by a constant times
$1/(n^3 (\log n)^c)$:
\[\begin{aligned}
\lambda_0 - \lambda_1 \gg  \frac{1}{n^3 (\log n)^c}
\end{aligned}\]
and the strong mixing time is bounded from above by $O(n^3 (\log n)^c)$:
\[
t_{\mix,1/(2|G|),\ell^\infty} \ll n^3 (\log n)^c,
\]
where $c$ and the implied constants are absolute.
\end{thm}
(Here, as is usual, ``$A\ll B$'' means ``$|A|\leq c B$ for some constant $c$''.
If we wish to emphasize that $c$ depends on some quantity $\delta$, we
write ``$A\ll_\delta B$''. An absolute constant is one that depends on no 
quantity, i.e., a constant that
 is truly a constant. For us, $A\ll B$ and $A = O(B)$ 
are synonyms.)

There is a bound on the spectral gap in terms of the diameter
(see, e.g., \cite[\S 3, Cor. 1]{MR1245303},
or the sources \cite{Aldous}, \cite{Babtech}, \cite{Gangthe},
\cite{Mohtech} given in \cite{MR1245303}) and also a standard bound on the
mixing time in terms of the spectral gap 
(as in, e.g., \cite[Thm.~5.1]{MR1395866}, or, for the $\ell^1$-norm, 
\cite[Thm.~12.3]{MR2466937}).
By means of such bounds, Theorem \ref{thm:diam} implies that
$\lambda_0 - \lambda_1 \gg 1/(n^4 (\log n)^{2c})$
and $t_{\text{mix},1/(2|G|),|G|\cdot \ell^\infty} \ll n^5 (\log n)^{2c}$. 
We obtain the stronger bounds in Thm.~\ref{thm:mix} by using 
most of the {\em proof} of Thm.~\ref{thm:diam}.

Both Thm.~\ref{thm:mix} and Thm.~\ref{thm:diam} rest in part on ideas
from \cite{BBS04} and \cite{Babai-Hayes} and in part on
the fact that, for $S = \{g,h,g^{-1},h^{-1}\}$, where $g$, $h$ are
random elements of $\Sym(n)$, the Schreier graph associated
to the action of $G=\langle S\rangle$ on $\ell$-tuples ($\ell$ constant) is almost certainly
an expander. This fact
was proven in \cite[Thm.~2.1]{MR1639748}. We thank B. Tsaban
for pointing this out to us; an earlier version of the present paper
contained a proof close to the one that can be found in \cite{MR1639748}.
We had been inspired by the proof for the case $\ell=1$ given in
\cite{BroSha}.

We should explain what we mean by ``the Schreier graph is almost certainly
an expander''. The 
{\em Schreier graph } $\Gamma(G\to X,S)$ ($G$ a group, $G\to X$ an action,
$S\subset G$, $S = S^{-1}$, $\langle S\rangle = G$) 
is the graph having $X$ as its
set of vertices and $\{(x,x^s): x\in X, s\in S\}$ as its set of vertices.
(As is common in the study of 
permutation groups, we write $x^s$ for the image $s(x)$
of $x$ under the action of $s$.) Given a Schreier graph, the normalized
adjacency matrix $\mathscr{A}$ is the operator on functions $f:X\to \mathbb{C}$
given by
\begin{equation}\label{eq:mashmi}
\mathscr{A} f(x) := \frac{1}{|S|} \sum_{h\in S} f(x^h) .
\end{equation}
Just as for a Cayley graph, the spectral gap of $\Gamma(G\to X,S)$ is the difference $\lambda_0-\lambda_1$
between the two largest eigenvalues $\lambda_0,\lambda_1$ ($\lambda_0=1$, $\lambda_0\geq \lambda_1$) of $\mathscr{A}$. (Since $S=S^{-1}$, the spectrum is real.)

What was proven in \cite{MR1639748} is that, for every $\ell$, there is a
$\delta>0$ such that the probability that $\lambda_0-\lambda_1\geq \delta$
(for the largest eigenvalues $\lambda_0$, $\lambda_1$ of 
the Schreier graph $\Gamma(G\to X,S)$, where $X$ is the set of
all $\ell$-tuples of distinct elements of $\{1,2,\dotsc,n\}$)
 tends to $1$ as $n\to \infty$.
Here, as we said before,
 $S = \{g,h,g^{-1},h^{-1}\}$ and $G = \langle S\rangle$,
where $g$, $h$ is a pair of random elements of $\Sym(n)$. We will use this
for $\ell=3$.


Lastly, let us remark that most of the arguments in this paper have an 
algorithmic flavor, in part inherited from \cite{BBS04}. This motivates
the following question: can theorem \ref{thm:diam} be made fully
algorithmic? That is, given random elements $g$, $h$ of $\Sym(n)$,
is it true that, with probability $1-o(1)$, for every $x\in \langle g,h\rangle$,
we can quickly find a word $w$ of length $O(n^2 (\log n)^c)$ such that
$w(g,h)=x$? We don't attempt to answer this question fully;
we sketch some ideas in Appendix \ref{sec:algori}.

\subsection{Relation to the previous literature}
The best bounds known for the problem addressed by
Thm.~\ref{thm:diam} were, successively, 
\[O(n^{(1/2+o(1)) \log n})\; \text{\cite{MR1208801}},
O(n^{8+o(1)})\; \text{\cite{Babai-Hayes}},\;\;\;\;
O(n^3 \log n)\; \text{\cite{MR2965280}}.\] The best worst-case bound known
(i.e., the best bound holding for all generating sets $S$)
is
$O(n^{O((\log n)^3 \log \log n)})$ \cite{Helfgott-Seress}.
Back in
\cite{MR2023654}, the kind of question addressed by both
Thm.~\ref{thm:diam} and Thm. \ref{thm:mix}
had been described as ``wide open'' (see 
\cite[Problem 8.11]{MR2023654} and the remarks immediately
following).

The bounds in \ref{thm:diam} and \ref{thm:mix} are close to the actual 
diameter and mixing times for at least some pairs $g$, $h$ (called
``slow shuffles'' in \cite[pp. 284--285]{MR2023654}). Take,
for instance, $g = (1 2 \dotsc n)$, $h = (1 2)$. Let $G = \Sym(n)$ and
$S = \{g,h,g^{-1},h^{-1}\}$. Then $\diam(G,S)$ is in the order of
$n^2$ (see, e.g., \cite{MR1022771}) and the mixing time $t_{\text{mix}}$ is in the
order of $n^3 \log n$ (\cite[p. 337]{MR2023654},
\cite[Thm. 1]{MR1987096}). The same bounds hold for related choices of $g$
and 
$h$, e.g., $g=(1 2 \dotsc n)$, $h = (1 2 \dotsc n-1)$ (the {\em Rudvalis
  shuffle};
see \cite{MR1987096}).

An algorithmic approach is proposed in \cite{MR2925926}; the algorithm
given there is shown to produce words of length $O(n^2 \log n)$
conditionally on a statement that remains unproven (the ``Minimal Cycle
Conjecture'').

In \cite{Diac}, Diaconis mentions ``an old conjecture [of his]'', which
states that the random walk on $\Alt(n)$ for {\em any} generating set
$\{g,h\}$ with two elements ``gets random in at most $n^3 \log n$ steps''
(i.e., has mixing time $O(n^3 \log n)$). 


\subsection{Acknowledgements}
H. Helfgott and A. Zuk regret to communicate that their coauthor,
\'A. Seress, passed away early in 2013. 

The authors would like to thank L. Saloff-Coste, M. Kassabov and I. Pak
for helpful discussions on mixing times, and B. Tsaban, for a
reference to the literature.

H. Helfgott and A. Zuk were supported by 
ANR Project Caesar No. ANR-12-BS01-0011.
H. Helfgott was also partially supported by the Adams Prize and the
Leverhulme Prize.

\section{Construction of small cycles}

We will construct $2$- and $3$-cycles as short words on $g$ and $h$.
The procedure goes back to \cite{BBS04}; we get better results because
our input -- namely, the fact that the Schreier graph on $\ell$-tuples
is almost certainly an expander --  is much stronger than the result
\cite[``Fact 2.1'']{BBS04} used in \cite{BBS04}.

\begin{prop}\label{prop:croute}
Let $g$, $h$ be two random elements of $\Sym(n)$. Then, with probability
$1-o(1)$, every $3$-cycle in $\Sym(n)$
can be written as a word of length $O(n (\log n)^c)$ in $g$ and $h$,
where
$c$ and the implied constant are absolute. 
\end{prop}
\begin{proof}
We will first show that there is at least one $2$- or $3$-cycle that
can be written as a word of length $O(n (\log n)^2)$ in $g$ and $h$.
We can assume that the Schreier graph corresponding to the 
action of $\{g,h,g^{-1},h^{-1}\}$ on the set $X$ of 
$3$-tuples of distinct elements
of $\lbrack 1,n\rbrack$
has a spectral gap of size at least $\delta$, 
since \cite[Thm.~2.1]{MR1639748} states that
this is the case with probability $1-o(1)$. Here $\delta>0$ is an absolute
constant.

Just as in the case
of a Cayley graph, the existence of a constant spectral gap means that
the Schreier graph has $(\epsilon,|X|\cdot \ell^\infty)$-mixing time 
$\ll_\delta \log |X|/\epsilon \leq \log n^3/\epsilon = O(\log n/\epsilon)$. 
(This is classical; see the expositions in the proofs of
e.g., \cite[Thm. 5.1]{MR1395866} (undirected graphs) and  
\cite[Lem. 4.1]{Helfgott-Seress}.) In other words, there is an absolute
constant $C$ such that, for every $\epsilon>0$ and every $k\geq 
C \log n/\epsilon$,
the probability $\Prob(\vec{v}^\sigma = \vec{w})$ that the outcome $\sigma$ of
a lazy random walk of length
$k$ on $g$, $h$ will take a given $\vec{v}\in X$ to a given $\vec{w}\in X$
satisfies
\begin{equation}\label{eq:hust}\frac{1-\epsilon}{|X|} \leq
\Prob(\vec{v}^\sigma=\vec{w}) \leq \frac{1+\epsilon}{|X|}.\end{equation}
This implies that, for any two triples of distinct elements
$(x,y_1,y_2), (x',y_1',y_2')\in X$,
\begin{equation}\label{eq:mordo}
\frac{1 - O(\epsilon)}{n-2}\leq  \Prob(x^\sigma = x') \leq \frac{1 
+ O(\epsilon)}{n-2}
,\end{equation}
where the implied constants are absolute.

By \cite[Lem. 6.2]{Babai-Hayes}, since $g$ and $h$ are random,
then, with probability $1-o(1)$, there is a
$j\leq 10 \log n$ such that either $v=h^j$ or $v=g h^j$ is an element
with a cycle of length $l\geq 3 n/4$. The argument in 
\cite[\S 6, proof of Thm.~ 2.2]{Babai-Hayes} gives us, moreover, that
$v^l\ne e$ with probability $1-O((\log n)/n^{1/4})$. Set $s = v^l$.

We can now follow the argument in \cite[\S 2]{BBS04}, 
only with much shorter walks. Let us give it in full for the sake of 
completeness.

Given $x\in \Sym(n)$, write $\supp(x)$ for the support of $x$, 
i.e., the set of elements of
$\lbrack 1,n\rbrack$ moved by $x$.
Choose $y_1,y_1'\in \supp(s)$, $y_2\notin \supp(s)$, $y_2' =
(y_1')^{s^{-1}}$. We define $\lbrack x,y\rbrack$ (the commutator)
to mean $x^{-1} y^{-1} x y$. (This definition is standard for permutation
groups). Let $\sigma\in \Sym(n)$ be such
that $y_i^\sigma = y_i'$ for $i=1,2$. Let $\tau = \sigma^{-1} s \sigma$.
The idea here is that
\[(y_1')^{s^{-1} \tau^{-1} s \tau} = (y_2')^{\tau^{-1} s \tau} = (y_2')^{s \tau} = 
(y_1')^\tau = y_1^{s \sigma} \ne y_1^\sigma = y_1',\]
and this assures us that $\lbrack s,\tau\rbrack$ cannot be the identity.
(This was the entire purpose of our definitions of $y_i$, $y_i'$ for
$i=1,2$, and of the condition $y_i^\sigma = y_i'$.)

We wish to show that $\lbrack s,\tau\rbrack$ has support much smaller than
that of $g$. Let us see -- $s$ is fixed and $\tau = \sigma^{-1} s \sigma$.
How shall we choose $\sigma$? Let $\sigma$ be the outcome of a lazy random
walk on $g$ and $h$
 of length $k\geq C \log n/\epsilon$. Let us
impose the condition that $y_i^\sigma = y_i'$ for $i=1,2$; it is easy
to see from (\ref{eq:hust}) that this happens with positive probability
(provided that $\epsilon>0$ is smaller than an absolute constant). 

Let $S = \supp(s)$. A brief case-by-case analysis (as in \cite[\S 2]{BBS04},
or as in the proof of Prop.~5.3 in the survey article \cite{Helfsur}) gives
us that
\[\supp(\lbrack s,\tau\rbrack) \subset (S\cap S^\sigma) \cup 
(S\cap S^\sigma)^s \cup (S\cap S^\sigma)^\tau\]
and so $|\supp(\lbrack s,\tau\rbrack)|\leq 3 |S\cap S^\sigma|$.
Since expected values are additive,
\[\begin{aligned}
\mathbb{E}(|S\cap S^\sigma| | y_i^\sigma = y_i') 
&= \sum_{x'\in S} \Prob(x'\in S^\sigma | y_i^\sigma=y_i')\\ &= 1 + 
\sum_{x'\in S} \mathop{\sum_{x\in S}}_{x\ne y_1,y_2} \Prob(x^\sigma = x' 
| y_i^\sigma=y_i')\\
&= 1+ \sum_{x'\in S} \mathop{\sum_{x\in S}}_{x\ne y_i} 
\frac{1+ O(\epsilon)}{n-2} = 1 + \frac{(1+ O(\epsilon)) |S|}{n-2} |S|
,\end{aligned}\]
where $i=1,2$. (We are using the assumptions that $y_1,y_1'\in S$ and 
$y_2\notin S$.)
We set $\epsilon$ small enough so $(1+O(\epsilon))/(n-2)$ here is less
than $7n /6$ (say), assuming (as we may) that $n$ is larger than an
absolute constant.

We conclude that there exists a 
$\sigma$ given as a word of length at most $k$ on $g$ and $h$ such that
$y_i^\sigma = y_i'$ for $i=1,2$ and $|S\cap S^\sigma| \leq 
1 + (7/6) |S|^2/n$. As we have seen, this implies that
$\lbrack s,\tau\rbrack$ is a non-identity element such that
\[|\supp(\lbrack s,\tau\rbrack)|\leq 3
\left(1 + \frac{7|S|}{6 n} |S|\right)\leq
3 + \frac{7 |S|}{2 n} |S| = 
3 + \frac{7}{8} |\supp(s)|.\]

Notice that $\lbrack s,\tau\rbrack$ is given by a word
on $g$ and $h$ whose length is at most $4$ times the length of the word
giving $s$, plus $4 k$.

We define $s_1 = \lbrack s,\tau\rbrack$ and iterate, constructing $s_2$,
$s_3$,\dots of decreasing support. We have $\supp(s_{j+1})\leq 3 + (7/2)
\supp(s_j)^2/n$, and so, already for some $j_0\ll \log \log n$, we have
$\supp(s_{j_0})\leq 3$,
where the implied
constants are absolute and  $s_{j_0}$  is not the identity. It is easy to check that $s_{j_0}$ is given as
a word of length at most $O(n (\log n)^c)$ on $g$ and $h$, where $c$ and
the implied constant are absolute.

Again, a random walk of length $k$ takes any $3$-tuple in $X$ to any
other $3$-tuple in $X$ with positive probability. Hence, we can express
any $3$-cycle (or $2$-cycle, if $|\supp(s_{j_0})|=2$)
as $r s_{j_0} r^{-1}$, where $r$ is a word of length
at most $k$ on $g$ and $h$. If $|\supp(s_{j_0})|=2$, note that every
$3$-cycle can be expressed as a product of two $2$-cycles. Hence, in
general, 
we conclude that every $3$-cycle in $\Sym(n)$ can be expressed as a word
of length at most $O(n (\log n)^c)$ on $g$ and $h$, where $c$ and
the implied constant are absolute.
\end{proof}

\section{Diameter, mixing times and spectral gaps}

It is easy to see that Proposition \ref{prop:croute} implies a bound
on the diameter of the Cayley graph.

\begin{proof}[Proof of theorem \ref{thm:diam}]
The diameter of $\Alt(n)$ with respect to the set of $3$-cycles is  $O(n)$. Hence, 
Proposition \ref{prop:croute} implies that, with probability $1$, every element of $\Alt(n)$ can be written as a word of length at most $O(n^2 (\log n)^c)$ on
$g$ and $h$, where $c$ and the implied constant are absolute. 

This implies, in particular, that $g$ and $h$ generate either $\Alt(n)$
or $\Sym(n)$. If $g$ and $h$ generate
$\Alt(n)$, we are done. If they generate $\Sym(n)$, 
then either $g$ or $h$ is in $\Sym(n)\setminus \Alt(n)$. Then every
element of $\Sym(n)$ can be written as a word of length at most $O(n^2 (\log n)^c)+1$ on $g$ and $h$, and so we are done, too.
\end{proof}

We will now see how Proposition \ref{prop:croute}  implies upper bounds on the 
mixing time and spectral gap for the Cayley graph $\Gamma(G,\{g,h,g^{-1},
h^{-1}\})$.
As we discussed in the introduction, 
these bounds are better than what one would obtain by proceeding from
the final result, Thm. \ref{thm:diam}, via comparison methods.


While getting a good bound on the spectral gap (and hence on the mixing time)
is slightly subtler than bounding the diameter,
the basic strategy is similar:

1) Solve the problem for the set of generators $A = \{ \text{$3$-cycles} \}$
of $\Alt(n)$ or the set of generators
\begin{equation}\label{eq:malga}
A = \{ \text{$3$-cycles} \} \cup  \{\text{one element of $\Sym(n)\setminus \Alt(n)$}\}\end{equation} of $\Sym(n)$. This, as we have
seen, is trivial when the problem consists in bounding the diameter. For the problem of bounding the spectral gap, the solution for $G = \Alt(n)$ and
$A = \{\text{$3$-cycles}\}$ is a computation
that we leave for Appendix \ref{sec:juice}; here, we will show how to deduce
from it the spectral gap for $G = \Sym(n)$ and $A$ as in (\ref{eq:malga}).

2) Use the fact that every 3-cycle can be written as a short word in $A$ 
(Proposition \ref{prop:croute}) to give a bound for the spectral
gap with respect to the generating set $A$. 
This was easy for the problem of bounding the diameter. 
To bound the spectral gap, we will use 
a known comparison technique (as in 
\cite{MR1245303}; see \cite[\S 13.5]{MR2466937} for the previous
history of the method).


We will show in Appendix \ref{sec:juice} that the spectral gap in
$\Alt(n)$ with respect to the set $C$ of all $3$-cycles is 
$\lambda_0 - \lambda_1 = 3/(n-1) > 1/n$ (Proposition \ref{prop:lambda}).
If $g$ and $h$ generate $\Sym(n)$ rather than $\Alt(n)$ it is not
enough to (a) prove a spectral gap for $\Alt(n)$ with respect to the set $C$
of $3$-cycles; we must actually (b) prove a spectral gap for $\Sym(n)$
with respect to the set $A = C \cup \{g\}$, where we
assume without loss of generality that $g\notin \Alt(n)$. Let us see how
(a) leads to (b).

Here and in what follows, given a finite set $X$, we write $L^2(X)$ for 
the space of functions $f:X\to \mathbb{C}$,
 equipped with the {\em unnormalized} $L^2$-norm
\[\|f\|_2 = \sqrt{\sum_{x\in X} |f(x)|^2}.\]

Let $M$ be the operator on $L^2(\Alt(n))$ defined by convolution with the
probability measure $p'$ that is uniformly distributed on the set $C\subset \Alt(n)$ of all $3$-cycles. By abuse of language, we also denote
by $M$ the operator on $L^2(\Sym(n))$ given by convolution with $p'$. 
(In other words, $MF = p'\ast F$ for
$F\in L^2(\Alt(n))$ and also for $F\in L^2(\Sym(n))$, with the convolution
being taken in $L^2(\Alt(n))$ and $L^2(\Sym(n))$, respectively.)

For $g\in \Sym(n)\setminus \Alt(n)$, consider the operator \[\widetilde{M}
= \frac{1}{2} (g + g^{-1}) M\] acting on $L^2(\Sym(n))$. (Here $hM$
is defined by $((hM)(F))(x) = 
MF(h^{-1} x)$; in other words, $hM$ is the composition of (1) the convolution with
the point measure $\mu_h$ at $h$ and (2) the operator $M$.)



\begin{prop}\label{prop:garna}
The spectral gap of $\widetilde{M}$ on $L^2(\Sym(n))$ is at least as large as
the spectral gap of $M$ on $L^2(\Alt(n))$.
\end{prop}
Since Prop.~\ref{prop:lambda} states that the spectral gap of $M$ on
$L^2(\Alt(n))$ is $3/(n-1)$, Prop.~\ref{prop:garna} implies that the spectral gap of $\widetilde{M}$ on $L^2(\Sym(n))$ is at least $3/(n-1)$.
\begin{proof}
The operator $\widetilde{M}$ is a convolution with a symmetric probability 
measure, namely, the average of the uniform probability measure on $g C$
and the uniform probability measure on $g^{-1} C = C g^{-1}$. (Here
$g^{-1} C = C g^{-1}$ because $C = g C g^{-1}$.) In particular, $\widetilde{M}$
is self-adjoint.

Let us first examine the action of $\widetilde{M}$ on eigenfunctions $f$ of $M$
with eigenvalue $1$. Any such function must be constant on $\Alt(n)$ (say
equal to $a$) and constant on $\Sym(n)\setminus \Alt(n)$ (say equal to $b$).
Since $M$ and $\widetilde{M} = (1/2) (g + g^{-1}) M$ commute (thanks to 
$g C g^{-1} = C = g^{-1} C g$), $f$ must also be an eigenfunction of 
$\widetilde{M}$ with eigenvalue $\lambda$, say. Then $\lambda a = b$ and 
$\lambda b = a$; hence $\lambda^2=1$, and so either $\lambda=1$ or $\lambda=-1$.
If $\lambda=1$, then $a=b$, and so $f$ is just a constant function on $\Sym(n)$.
(If $\lambda=-1$, then $a=-b$. Obviously, we need not worry about $\lambda=-1$,
since then $1-\lambda=1-(-1)\geq 2$, and $2$ is certainly at least as large
as the spectral gap of $M$ on $L^2(\Alt(n))$.)

Now consider the action of $\widetilde{M}$ on the space $H$ of functions $f$
on $\Sym(n)$ such that $f|_{\Alt(n)}$ is orthogonal to constant functions on
$\Alt(n)$ and $f|_{\Sym(n)\setminus \Alt(n)}$ 
is orthogonal to constant functions on $\Sym(n)\setminus \Alt(n)$.  
Then, for every $f\in H$, we know that $\|Mf\|_2\leq (1-\delta) \|f\|_2$, where
$\delta$ is the spectral gap of $M$ on $L^2(\Alt(n))$. Now, convolution with
$(1/2) (\mu_g + \mu_{g^{-1}})$ is an operator of norm at most $1$, simply
because convolution with any probability measure is an operator of norm
at most $1$. Hence
\[\|\widetilde {M} f\|_2 = \left\|\frac{1}{2} (\mu_g + \mu_{g^{-1}})(M f) \right\|_2 \leq \|M f\|_2 \leq (1-\delta) \|f\|_2,\]
proving our statement.
\end{proof}

Let us now examine the mixing time. We now
 work with the normalized $\ell^p$-norm:
\begin{equation}\label{eq:normell}
|f|_p = \left(\frac{1}{|G|} \sum_x |f(x)|\right)^{1/p}.\end{equation}
(This is consistent with the normalization of the $\ell^1$-norm in the 
introduction.) We know from (\ref{eq:l2mix}) that
$t_{\text{mix},\epsilon/|G|,\ell^2} \ll_\epsilon n \log n$ 
for $G=\Alt(n)$
and $S$ equal to $C$, the set of all $3$-cycles. Consider now
$G=\Sym(n)$ and a random walk with respect to the distribution
\begin{equation}\label{eq:luchti}
\mu'(x) = \frac{1}{2} \mu(x) +  \frac{1}{2} \mu(g^{-1} x),\end{equation}
where $g\in \Sym(n)\setminus \Alt(n)$ and
$\mu = 1_{\{e\}}/2 + 1_C/(2 |C|)$.
The class $C$ is invariant under conjugation by $g$
(or by any other element). Hence, a random walk of length $k$
with respect to $\mu'$
gives the result $g^r x$, where $x$ is the outcome of
a random walk of length $k$ with respect to $\mu$ (i.e., a lazy random walk 
of length $k$ with respect to $C$) and $r$ is a random integer that is both
independent of the random walk and of equidistributed parity. In other
words,
\[(\mu')^{(k)} = \sum_{r=0}^k p_r \mu^{(k)}(g^{-r} x),\]
where $\sum_{\text{$r$ odd}} p_r = \sum_{\text{$r$ even}} p_r = 1/2$.
It is easy to check that this implies that,
 for any left-invariant distance function $d$ (such as $d = \ell^2$),
\begin{equation}\label{eq:beeth}\begin{aligned}
d\left((\mu')^{(k)},\frac{1_G}{|G|}\right)
\leq d\left(\mu^{(k)},\frac{1_{\Alt(n)}}{|\Alt(n)|}\right).\end{aligned}
\end{equation}
(We are using the fact that $g^r \Alt(n) = \Alt(n)$ for $r$ even, and
$g^r \Alt(n) = \Sym(n) \setminus \Alt(n)$ for $r$ odd.)

Before we proceed, we should make a brief remark on how $\ell^p$-mixing times
relate to each other for different $p$. It is well-known that, if 
$\ell^p$ norms are defined as in (\ref{eq:normell}) and $p\leq q$, then
$|\cdot|_p\leq |\cdot|_q$ for $p\leq q$; this is a special case of
Jensen's inequality (see, e.g., \cite[Thm.~3.3]{MR924157}). 
This allows us to use $\ell^q$-mixing times to bound $\ell^p$-mixing times
for $q\geq p$.
There is a way (also well-known) to use $\ell^2$-mixing times
to bound $\ell^\infty$-mixing times: for any measure $\mu$ on any finite
group $G$ and any $x\in G$,
\[\begin{aligned}
\mu^{(2k)}(x) &=
\sum_{g\in G} \mu^{(k)}(g) \mu^{(k)}(g^{-1} x) \\
&= \sum_{g\in G} \frac{1}{|G|}\cdot \mu^{(k)}(g^{-1} x) 
+ 
\sum_{g\in G} \left(\mu^{(k)}(g) - \frac{1}{|G|}\right) 
\left(\mu^{(k)}(g^{-1} x) - \frac{1}{|G|}\right)\\
&+
\sum_{g\in G} \left(\mu^{(k)}(g) - \frac{1}{|G|}\right) \frac{1}{|G|}
\leq \frac{1}{|G|} + |G| \left|\mu^{(k)} - \frac{1}{|G|}\right|_2^2,
\end{aligned}\]
i.e., $|\mu^{(2k)} - 1/|G||_\infty \leq |G| |\mu^{(k)} - 1/|G||_2^2$. 
This implies that
\begin{equation}\label{eq:argu}
t_{\mix,\epsilon^2/|G|,\ell^\infty} \leq 2 t_{\mix,\epsilon/|G|,\ell^2}.\end{equation}
\begin{proof}[Proof of Theorem \ref{thm:mix}] 
We would like to estimate the spectral gap for the random 
choice of generators by comparing it with the spectral gap for 
either $M$ or $\widetilde{M}$.
We will use a comparison technique from \cite{MR1245303}.

Let $S$ be a symmetric set of generators of $G$. For $y \in G$ and 
$s \in S$ we define $N(s,y)$ to be the number of times $s$ occurs in a chosen
expression for $y$ as a product of elements of $S$.
Let $p$ and $p'$ be symmetric probability distributions on $G$. 
Suppose that the support of $p$ contains $S$.

Consider the following quantity:
\begin{equation}\label{eq:bethle}
A  = \max_{s \in S}   \frac{1}{p(S)}   \sum_{y \in G}  |y|  N(s,y)  p'(y).
\end{equation}
We can use $A$ to compare 
the spectral gap $\delta(p) = \lambda_0(p) - \lambda_1(p)$ 
for the convolution by $p$
with the spectral gap $\delta(p) = \lambda_0(p') - \lambda_1(p')$ 
for the convolution by $p'$ \cite{MR1245303}:
$$\delta(p)   \geq \frac{1}{A} \delta(p').$$ 
We set $S = \{g,h,g^{-1},h^{-1}\}$,
where $g,h\in \Sym(n)$ are chosen randomly. 
If
$g,h\in \Alt(n)$, 
we let 
$p'$ be uniformly supported 
on the set $C$ of all 3-cycles; otherwise, we assume without
loss of generality that $g\notin \Alt(n)$, and we let $p'$ be the average of
the uniform probability distribution on $g C$ and the uniform probability 
distribution on $g^{-1} C$. We set $p = \mu$, with $\mu$
given as in (\ref{eq:ocio}).
 From Proposition \ref{prop:croute}, we get that,
with probability $1-o(1)$, $A\ll (n (\log n)^c)^2$,
where $c$ and the implied constant are absolute. 
Indeed, Prop.~\ref{prop:croute} assures us that $G=\langle g,h\rangle$
is either $\Alt(n)$ or $\Sym(n)$; even more importantly, Prop.~\ref{prop:croute}
tells us that the diameter of $G$ with respect to $S$ is $O(n (\log n)^c)$
if $G=\Alt(n)$, and also that it is $O(n (\log n)^c)+1 = O(n (\log n)^c)$ if
$G=\Sym(n)$. Since we can bound both $|y|$ and $ N(s,y)$ by the diameter of $G$
with respect to $S$, this means that $|y|$ and $N(s,y)$ are both $O(n (\log n)^c)$. Hence $A\ll (n (\log n)^c)^2$.

We know that
\[\delta(p') \geq \frac{3}{n-1}\]
by Prop.~\ref{prop:lambda} if $G=\Alt(n)$, and by Prop.~\ref{prop:lambda}
and Prop.~\ref{prop:garna}. We conclude that
\[\lambda_0(p) - \lambda_1(p) = \delta(p) \geq \frac{1}{A} \cdot \frac{3}{n-1}
\gg \frac{1}{(n (\log n)^c)^2} \cdot \frac{3}{n} > \frac{1}{n^3 (\log n)^{2 c}}.
\]

We could bound the mixing time by $O\left(n^4 (\log n)^{O(1)}\right)$
using this spectral gap estimate. We will do better by working with
mixing times directly.

Again, we work with $S = \{g,h,g^{-1},h^{-1}\}$.
If $g,h\in \Alt(n)$, we let $p'=\mu$, where $\mu=1_{e}/2+1_C/(2|C|)$, for $C$
the set $C$ of all $3$-cycles; otherwise, we assume w.l.o.g.\ that 
$g\notin \Alt(n)$, and we let $p' = \mu'$, where $\mu'$ is as in 
(\ref{eq:luchti}). We let $p = 1_{e}/2 + 1_{|S|}/(2 |S|)$, just as before.
By the same argument as above, the quantity $A$ defined in 
(\ref{eq:bethle}) is $\leq C (n (\log n)^c)^2$ (where $C$, $c$ are 
absolute constants) 
with probability $1-o(1)$
(for $g$ and $h$ random). Now, a comparison result
in \cite{MR1245303} allows us to finish the task. We will quote the result
as stated in \cite[Thm. 10.2 and 10.3]{MR2023654}:
\[\left|p^{(k)} - \frac{1}{|G|}\right|_2^2 \leq |G| e^{-k/2 A} + 
\left|(p')^{(\lfloor k/2 |A|\rceil)} - \frac{1}{|G|}\right|_2^2.
\]
(Some of the terms in \cite[Thm. 10.2]{MR2023654} disappear due to the
fact that the spectrum of $p$ is non-negative.)
At the same time, by (\ref{eq:l2mix}) and (\ref{eq:beeth}),
\[\left|(p')^{(k')} - \frac{1}{|G|}\right|_2 \leq \frac{\epsilon}{|G|}\]
for any $k'\geq C_\epsilon n \log n$, where $C_\epsilon$ depends only on
$\epsilon$. We set \[k = \lceil 2 |A| \cdot \max(4,C_\epsilon) n \log n
\rceil,\]
and we obtain that 
\[\left|p^{(k)} - \frac{1}{|G|}\right|_2^2 \leq \frac{1}{|G|^3} + 
\frac{\epsilon^2}{|G|^2}.\]
Setting $\epsilon = 1/2$, we get that $|p^{(k)}-1/|G||_2 \leq 
(2/3) |G|$ (say) for $n$ larger than a constant. Hence, by (\ref{eq:argu}),
\[t_{\mix,\frac{1}{2|G|},\ell^\infty} \leq 
2 t_{\mix,\frac{2/3}{|G|},\ell^2} \leq 2 k \ll n^3 (\log n)^{2 c + 1}.\]

\end{proof}

\appendix

\section{The spectral gap and the mixing time of $\Alt(n)$ with respect to $3$-cycles}\label{sec:juice}

We need to bound the spectral gap of $\Alt(n)$ with respect to the generating
set consisting of all $3$-cycles in $\Alt(n)$. We will actually compute the
spectral gap exactly.

Let us first review the literature briefly.
It was computed in \cite[p. 175]{MR626813} that,
for $G=\Sym(n)$ and $S$ equal to the set $\{(i j): 1\leq i<j\leq n\} \cup
\{e\}$, with the identity $e$ being given weight $1/n$, the eigenvalue gap
$\lambda_0-\lambda_1$ is $2/n$. We will also need a result for $S$ equal
to the set of $3$-cycles. For such an $S$, we could deduce a bound of
$\lambda_0 - \lambda_1 \gg 1/(n \log n)$ from \cite{MR2884874}.
Here we will follow the approach of \cite{MR626813} to show that
$\lambda_0 - \lambda_1 = 3/(n-1)$.

We will now remind the
reader of some basics in the representation theory of finite groups.
A {\em representation} $\rho$ of a finite group $G$ 
is a homomorphism from $G$ to the group of invertible
linear operators of vector space $V$. We write
$d_\rho$ for the dimension of $V$; we will consider only the case of
$V$ finite-dimensional.

A representation $\rho$ 
is said to be {\em irreducible} if there is no non-trivial $\rho$ invariant
subspace of $V$. From now on, $V$ will be a vector space over $\mathbb{C}$. 
 Schur's lemma states that a linear operator from $V$ to $V$ which commutes with
an irreducible representation is a multiple of the identity.
Two representations $\rho$, $\rho'$ are {\em equivalent} if they are conjugates of each other, i.e.,
if there is an isomorphism $\phi:V\to W$ such that $\rho'(g) = 
\phi \circ \rho(g) \circ \phi^{-1}$.

Given a representation $\rho$ and a function $p$ on $G$, we
define the Fourier 
transform of $p$ by
\[\rho(p) =      \sum_{\gamma  \in G}    p(\gamma)   \rho (\gamma),\]
which is an endomorphism from $V$ to $V$. As usual,
the Fourier transform transforms a convolution into a multiplication:
$\rho(p_1 \ast p_2) = \rho(p_1) \rho(p_2) = \rho(p_2) \circ \rho(p_1)$.

The character $\chi_{\rho}:G\to \mathbb{C}$ of a 
representation $\rho$ is defined by
$\chi_{\rho} (  \gamma)  = \tr  (\rho ( \gamma))$.
Characters are constant on conjugacy classes.

Now consider any finite group $G$ and any irreducible representation $\rho$
of $G$.
Let $p$ be a function from $G$ to the complex numbers which is constant 
on each conjugacy class. Put the conjugacy classes in some arbitrary order.
Let $p_i$ be the value of $p$ on the $i$-th conjugacy class,
$n_i$ the cardinality of the $i$-th conjugacy class, and $\chi_{\rho,i}$ the value of 
$\chi_{\rho}$ on the $i$-th conjugacy class. Then
\begin{equation} \label{eq:irr}
  \rho (p)  =  \left(\frac{1}{d_\rho}
   \sum_i    p_i    n_i \chi_{\rho,i}    \right)
\text{Id}.
\end{equation}

Indeed 
$ \rho ( p ) = \sum_{\gamma}   p(\gamma)  \rho(\gamma)$ can be written as
$  \sum_i   p_i M_i$ where $M_i$ is the sum of $\rho(\gamma)$ over the $i$-th
conjugacy class. By the definition of the conjugacy class each matrix $M_i$
commutes  with $\rho ( \gamma)$. Since $\rho$ is irreducible, Schur's lemma
tells us that $M_i = c_i \text{Id}$. The trace of $M_i$ is equal both to
$n_i \chi_{\rho,i}$ and to $c_i d_{\rho}$; this gives the above formula.
  
\medskip

The left regular representation $\lambda$ is defined as follows.
  For $f \in \ell^2 (G)$,
\[\lambda (\gamma) (f) (\gamma')  = f (\gamma^{-1} \gamma ').\]

The action of $\lambda(p)$ on $\ell^2 (G)$ corresponds to a convolution by $p$.
Indeed, for $f \in \ell^2(G)$
$$ \lambda(p)  f (\gamma)   =   \sum_{\nu \in G}   \lambda (\nu) p (\nu) f(\gamma) =
 \sum_{\nu \in G} p (\nu) f(  \nu^{-1}\gamma)  = p \ast f (\gamma).
$$

The regular representation (and more generally any representation) 
can be decomposed into irreducible representations -- that is to say, it can be
written as a direct sum of irreducible representations.  Every 
irreducible representation appears in the
decomposition of the regular representation. Therefore, it follows from (\ref{eq:irr}) that, for $p$ constant on conjugacy classes (for instance, equidistributed
on $k$ cycles), the eigenvalues of the convolution operator 
$f\mapsto p\ast f$ on $\ell^2(G)$
are precisely the values of
\begin{equation}\label{eq:malgo}
\frac{1}{d_{\rho}}   \sum_i    p_i    n_i \chi_{\rho,i}\end{equation}
as $\rho$ ranges over the irreducible representations of $G$.

We will apply this theory to symmetric groups  $G=\Sym(n)$. We let
$p$ be the uniform probability measure on the set $C$ of all $3$-cycles.
Since $C$ generates only the alternating group $\Alt(n)$, 
which is a subgroup of index 2 in $\Sym(n)$, and we are using 
the representation theory of $\Sym(n)$,
this will result in doubling the multiplicity of all eigenvalues.
We can see this explicitly as follows. Let $f:\Alt(n)\to \mathbb{C}$
be an eigenfunction of the action of $p$ within the 
left regular representation of $\Alt(n)$ (i.e., $C f = \lambda f$, where
$Cf$ is defined by $C f(h)\mapsto \sum_g p(g) f(g h)$). Then $f$
defines two eigenfunctions of the action of $p$ within the left regular
representation of $\Sym(n)$, both of them restricting to $f$ on $\Alt(n)$:
given a fixed $s\in \Sym(n)\setminus \Alt(n)$, we let
$f(g s) = f(g)$ for $g\in \Alt(n)$ to define one of the eigenfunctions,
and $f(g s) = - f (g)$ to define the other one. 

(We can also see the doubling of the multiplicity of all eigenvalues
more abstractly, by using a result such as 
\cite[\S 5, Prop.~5.1]{MR1153249}.)

By a {\em partition} $\lambda = (\lambda_1 , ...  , \lambda_k)$ of $n$
we mean a non-increasing sequence of positive integers $\lambda_j$ with sum $n$.

It is a fundamental fact from the representation theory of $\Sym(n)$ that the irreducible
representations of $\Sym(n)$ are in one to one correspondence with partitions of $n$.

Let $p$ be the uniform probability measure on the set $C\subset \Alt(n)$ of
all $3$-cycles. Then, by (\ref{eq:irr}),
\[\rho(p) =  \frac{\chi_{\rho}(\sigma)}{d_{\rho}} \Id,\]
where $\sigma$ is a 3-cycle. Therefore, the eigenvalues of $p$ are
\begin{equation}\label{eq:moses} \chi_{\rho} ( \sigma)   /  d_{\rho}
\end{equation}
as $\rho$ ranges over all irreducible representations of $\Alt(n)$.

By a computation of Frobenius (as in \cite[(5.2)]{ingram}), 
\begin{equation}\label{eq:sampo}\frac{\chi_{\rho} ( \sigma)}{d_{\rho}}  =
\frac{M_3}{2n(n-1)(n-2)}  -   \frac{3}{2(n-2)}\end{equation}
for \begin{equation}\label{eq:kantele}
M_3 = \sum _{j=1}^k     \left( (\lambda_j   -  j  )   (\lambda_j   -  j  +1)  (2 \lambda_j  - 2j  +1)   +  j(j-1)(2j-1)\right) ,\end{equation}
where $(\lambda_1 , \dotsc , \lambda_k)$ is the partition corresponding  to $\rho$.

There is a following partial order on partitions of $n$. Let
$\lambda = ( \lambda_1, ... , \lambda_k)$ and 
$\lambda' = ( \lambda'_1, ... , \lambda'_{k'})$ be partitions of $n$. We define
$\lambda \geq \lambda'$  if $k \leq k'$ and
$\lambda_1  \geq \lambda'_1$, $\lambda_1  + \lambda_2  \geq  \lambda'_1+ \lambda'_2$, ...,
$\lambda_1  +.. + \lambda_k  \geq  \lambda'_1+ ...+  \lambda'_k$.

We say that a partition $\lambda'$ is obtained from a partition $\lambda$ by a single switch
if for some indices $a<b$, $\lambda_a = \lambda'_a +1$, $\lambda_b = \lambda'_b -1$,
and 
$\lambda_j$ and $\lambda'_j$ coincide for all other indices. (If
$\lambda'_b=1$, then the partition $\lambda$ simply ends at $k=k'-1$;
otherwise,
$k=k'$.)

It is not difficult to see that for any partitions  $\lambda \geq \lambda'$ of $n$
there is decreasing sequence of partitions obtained by a sequence of
switches starting at 
$\lambda$ and ending at  $\lambda'$.

%
%
%
%
%
%
%
%
%
%
%
%
%


\begin{lem} \label{sgap}
Consider two partitions $\lambda > \lambda'$ which differ by a single switch, i.e.
for some indices $a <b$, $\lambda_a  = \lambda'_{a}+1$, $\lambda_b  = \lambda'_b -1$,
and for all other indices $\lambda_j$ and $\lambda'_j$ coincide. Then
 the value of $M_3$ for $\lambda$ minus the value of $M_3$ for
$\lambda'$ equals
$$6   (   (  \lambda'_a +1 - a)^2      -  (\lambda'_b -b)^2   ).$$
\end{lem}

\begin{proof}
Consider the expression (\ref{eq:kantele}) defining $M_3$. 
When one makes a switch there is
a change in its value for two values for $j$, namely $a$ and $b$.

For $a$ the difference is 
$$( \lambda'_a - a +1) 
( \lambda'_a - a + 2) 
(2 \lambda'_a - 2a + 3)   -
( \lambda'_a - a) 
( \lambda'_a - a + 1) 
(2 \lambda'_a - 2a + 1),
$$
which equals $ 6( \lambda'_a - a + 1)^2$. In the same way,
the difference for $b$ is equal to
$6( \lambda'_b - b)^2$ (and, in particular, it is $6(1-b)^2$ when $\lambda_b'=1$).
\end{proof}

To any partition $\lambda= (\lambda_1, \ldots , \lambda_k)$ one can associate 
a Young diagram, which consists of $n$ squares.
The first column of the diagram consists of $\lambda_1$ squares, the second of 
$\lambda_2$ and so on. In the Young diagram correspond to the conjugate representation,
the first row of the diagram consists of $\lambda_1$ squares, the second of 
$\lambda_2$ and so on. 
In other words, the diagram is flipped so that columns become rows. It is
easy to see from (\ref{eq:irr})
 that the eigenvalues corresponding to conjugate representations are the same.

\begin{lem} \label{snap}
Consider the action of the measure supported uniformly on $3$-cycles
in the regular representation of $\Alt(n)$.
The largest eigenvalue 1 corresponds to 
partitions $(n)$ and $(1,1, \ldots,1)$. The second 
largest eigenvalue corresponds to partitions $(n-1,1)$ and $(2,1, \ldots,1)$.
\end{lem}


\begin{proof}[Proof of Lemma \ref{snap}]
For any partition $\lambda= (\lambda_1, \ldots , \lambda_k)$ and its conjugate 
 $\lambda'= (\lambda'_1, \ldots , \lambda'_{k'})$ we have
$\lambda'_1 = k$ and $\lambda_1=k'$.

From any partition we can obtain the partition $(n)$ by doing a sequence of
inverse switches of the form 
$$ ( \lambda_1', \ldots, \lambda_k')  \rightarrow  ( \lambda_1' + 1, \ldots,
\lambda_k' - 1)$$
in case $\lambda_k' \geq 2$, or
$$ ( \lambda_1', \ldots, \lambda_{k-1}', 1)  \rightarrow  ( \lambda_1' + 1,
\ldots, \lambda_{k -1}')$$ if $\lambda_k'=1$.

It follows from Lemma \ref{sgap} that a single switch will increase the eigenvalue as soon as
$\lambda_1' \geq k$.
This condition is satisfied either 
for a partition or its conjugate; moreover, it
 is preserved by the inverse switches described before. 
As conjugation does not change the value of the eigenvalue, by taking if necessary a
conjugate, we can suppose that we consider a partition for which there is sequence of 
switches ending up at the partition $(n)$, each of them increasing the
eigenvalue. The one before the last partition in this sequence is $(n-1,1)$. 
\end{proof}

\begin{prop} \label{prop:lambda}
The spectral gap in the regular representation of $\Alt(n)$ 
for the measure supported uniformly on $3$-cycles is $  3/(n-1).$
\end{prop}
This is, of course, the same as the spectral gap for the Cayley graph of
$\Alt(n)$ with respect to the set of generators consisting of all $3$-cycles.
\begin{proof}
By Lemma \ref{snap}, the spectral gap equals the difference in the
eigenvalues corresponding to  
partitions $(n)$ and $(n-1,1)$.
A simple computation starting from (\ref{eq:sampo})
shows that the eigenvalue associated  to $(n-1,1)$ is $1 - 3/(n-1)$,
whereas the eigenvalue associated to $(n)$ is, naturally, $1$.
Thus, the spectral gap is $3/(n-1)$.
\end{proof}

\begin{center}
* * *
\end{center}

It remains to examine the mixing time with respect to $3$-cycles.
In brief -- it is enough to get a bound for the $(\epsilon/|G|,\ell^2)$-mixing
time, since that can be used to bound the
$(\epsilon/|G|,\ell^\infty)$-mixing 
time (see (\ref{eq:argu})).
 Much of the literature on the subject gives $\ell^1$-mixing times,
which are weaker. 

(We scale $\ell^p$ norms as in (\ref{eq:normell}) 
-- that is, we are working with $G$ as a space of measure $1$.
Recall that, on a space
of measure $1$, $|\cdot |_p\leq |\cdot|_q$ for $p\leq q$, and so the
$(\epsilon,\ell^p)$-mixing time is at most the $(\epsilon,\ell^q)$-mixing time
for $p\leq q$.)

Fortunately, \cite{MR1799892} and \cite{MR1400314} both bound
$\ell^1$-mixing times by
the $(\epsilon/|G|,\ell^2)$-mixing time (via Cauchy-Schwarz); most
of the work in \cite{MR1799892} and \cite{MR1400314}
goes into bounding the $(\epsilon/|G|,\ell^2)$-mixing time, or, what
amounts to the same, bounding the sum $\sum_\rho d_\rho^2
|r_\rho(C)|^{2k}$ in \cite [(10)]{MR1400314}
(appearing with slightly different
notation in Remark 2.5 in \cite{MR1799892}). 
The bound they obtain (following the approach in \cite{MR626813}) is
\begin{equation}\label{eq:l2mix}
t_{\text{mix},\epsilon/|G|,\ell^2} \ll_\epsilon n \log n\end{equation}
for $G=\Alt(n)$ and $S$ the class of $3$-cycles. (In fact, 
\cite{MR1799892} gets the
stronger bound \[t_{\text{mix},\epsilon/|G|,\ell^2} = n \left(\frac{\log n}{3} +
\frac{\log(1/\epsilon)}{2} + O(1)\right),\]
 whereas \cite{MR1400314} proves a result for arbitrary
conjugacy classes.) See also Theorem 9.4 in the survey \cite{MR2023654}.

\section{Algorithmic remarks}\label{sec:algori}
Let $g$ and $h$ be two random elements of $\Sym(n)$. With probability
$1-o(1)$, the Schreier graph corresponding to the action of $S = \{g,h,g^{-1},
h^{-1}\}$ on the set of $3$-tuples of distinct elements of $\lbrack 1,n\rbrack$
is an expander graph. Suppose from now on that this is the case. 

We have shown that, in such a case, the diameter of $G$ with respect to $S$
is $O(n^2 (\log n)^c)$. In other words, for every $\pi\in \langle S\rangle$,
there is a word $w$ of length $O(n^2 (\log n)^c)$ such that $\pi = w(g,h)$. The question
is: can such a word $w$ always be found quickly?

First of all, we have to define our goals, i.e., what is meant by ``quickly''.
It would seem at first sight that we cannot hope for an algorithm taking
less time than the length of $w$, since, in general, it takes time proportional
to the length of $w$ to write down $w$. However, the words $w$ that the proof
of Thm.~\ref{thm:diam} yields are of a very special sort, in that they involve
high powers. To be precise, every word $w$ we find is of the form
\[ w(g,h) = v(g,h,u(g,h)^l),\]
where $v$ is a word of length $O\left(n (\log n)^c\right)$, $u$ is a word
of length $O(\log n)$ and $l<n$. Such a word can be written using 
$O\left(n (\log n)^{O(1)}\right)$ symbols; thus, it is not ruled out {\em 
a priori}
that there may be an algorithm that finds the word in 
$O\left(n (\log n)^{O(1)}\right)$ steps.
This is so even if we assume -- as we shall -- that just multiplying two
elements of $\Sym(n)$ takes time $O(n)$.

Let us explain how the proof we have given strongly suggests a way to
construct just such an algorithm. First of all, the algorithm in
the \cite{BBS04} (and hence that in the proof of Prop.~\ref{prop:croute})
 is in essence algorithmic. There seems to be only one problematic spot:
while it would seem that most (and not just some) random walks make the 
argument work, this stops being the case when we come to the point
in the proof where we
fix the condition $y_2' = (y_1')^{s^{-1}}$. If, on the other hand,
 we prefer not to impose this
condition, we can no longer guarantee that $\lbrack s,\tau\rbrack$ have
non-zero support. 

It seems possible to do without the condition
$y_2' = (y_1')^{s^{-1}}$ while also constructing
elements with non-zero support: (a) it is possible to assume that 
the length $l$ of the long cycle is smaller than $(1-\epsilon) l$
by an appropriate use of \cite{Babai-Hayes};
(b) we can choose to define $s_{i+1} = \lbrack s_i, \sigma^{-1} s_{j_i} \sigma
\rbrack$ for an appropriate $j_i\leq i$, rather than always use $j_i=i$.
Thus modified, the proof of Prop.~\ref{prop:croute} should give a permutation
$\varkappa$ 
of bounded support, as a word of length $O(n (\log n)^{O(1)})$, in time
$O(n (\log n)^{O(1)})$.

What would remain to do would be to show how to express every permutation
in $\Alt(n)$ or $\Sym(n)$ as a word of length $O(n (\log n)^{O(1)})$ in
$g$, $h$ and $\varkappa$, to be found in time  $O(n (\log n)^{O(1)})$. This
is not completely trivial even when $\varkappa$ is a $3$-cycle, in that
the long cycle of the element $v$ (constructed in the proof of Prop.~\ref{prop:croute}) has to be used
in conjunction with the effect of a random walk on $g$ and $h$.
For $\varkappa$ general, the matter is even less obvious.

\begin{center}
* * *
\end{center}

If all one aims for is running time $O(n^2 (\log n)^{O(1)})$, the problem
becomes considerably simpler. In brief -- one can use the algorithm
implicit in the proof of Prop.~\ref{prop:croute} essentially as it stands: 
at the
problematic spot, we can choose $y_1' \in \supp(s)$ arbitrarily, and fix
$y_2' = (y_1')^{s^{-1}}$ as before; then any choice of $y_1\in \supp(s)$,
$y_2\not\in \supp(s)$ is valid -- and there are plenty of such choices.
Summing over them, we see that the probability that 
$(y_1')^{\sigma^{-1}}\in \supp(s)$, $(y_2')^{\sigma^{-1}} \not\in \supp(s)$
is $\gg 1/n$ even in the worst case (which is the case of $|\supp(s)|$
bounded). Thus, we just need to keep generating $\sigma$ (at most
$O(n (\log n)^{O(1)})$ times) until we succeed in finding an $\sigma$ satisfying
$(y_1')^{\sigma^{-1}}\in \supp(s)$, $(y_2')^{\sigma^{-1}} \not\in \supp(s)$
and $|S\cap S^\sigma|\geq (7/6) |S|^2/n$ (with probability
$\geq 1 - 1/n^A$, $A$ arbitrary). 

We follow the rest of the proof of Prop.~\ref{prop:croute}, and we obtain
a $3$-cycle $\kappa$ as a word of length $O(n (\log n)^{c})$, 
in time $O(n^2 (\log n)^{O(1)})$, with probability 
$\geq 1 - \epsilon$; we can bring the probability arbitrarily close to $1$
by repeating the procedure. 
(If $\langle \{g,h\}\rangle = \Sym(n)$, it could happen that we actually
construct a $2$-cycle, rather than a $3$-cycle; in that case; the procedure
is still essentially what we are about to outline, only simpler.)
Once we have the $3$-cycle $\kappa$, one issue remains:
how do we use it to construct all the $3$-cycles we need, as words of
length $O(n (\log n)^{c})$, in time $O(n^2 (\log n)^c)$ (in total) or less?

It would not make sense to construct all $3$-cycles, since there are
$\gg n^3$ of them -- meaning they could not be constructed in time less
than $O(n^3)$. Instead, we start by writing the permutation $\pi$ we are given
as a product of $3$-cycles. (The way to do it is easy and
 well-known.) Our task is
 to express each one of those $O(n)$ cycles as a word of length
$O(n (\log n)^{c})$. Note it would not make sense to simply conjugate our
$3$-cycle $\kappa$ by short random walks, and store the result every time
it happens to be one of the $3$-cycles we need: this would take $\gg n^3$
repetitions. Rather, let us see a way to generate any $3$-cycle we need
rapidly, after a little initial preparation. The way to do this is to use
both our ability to scramble elements by means of random walks, and the
fact that we can use the long cycle in $v$ to shift $3$-cycles around. Let
us see how.

We can assume without loss of generality that the long cycle is labelled
\[(1\; 2 \dotsc l)\]
Recall that $l\geq 3 n/4$.
(Any bound of the form $l\gg n$ would do.)
We can also assume that $\supp \kappa \subset\{1,2,\dotsc,l\}$: if this is
not the case, we are still fine, since, for the outcome $\sigma$ of a short
random walk, $(\supp \kappa)^\sigma \subset \{1,2,\dotsc,l\}$ with
probability
$\geq (l/n)^3-o(1)$, and so, if we take a small number ($\ll \log n$) of
random walks, it is almost certain (probability $1-O\left(n^{-C}\right)$)
that one of them will give us a $\sigma$ such that
\[(\supp \kappa)^\sigma \subset \{1,2,\dotsc,l\}.\]
Notice that $\supp \sigma^{-1} \kappa \sigma = (\supp \kappa)^\sigma$. 
We redefine $\kappa$ to be $\sigma^{-1} \kappa \sigma$, and so we obtain
$\supp \kappa \subset\{1,2,\dotsc,l\}$ after all. We can assume w.l.o.g.\
that $\kappa = (1\; a\; b)$, where $1\leq a,b\leq l$.

First, let us see how to construct a 3-cycle of the form $(1\; 2\; x)$.
This is done as follows. Taking expected values and variances, it is
easy to show that, for the result $\gamma$ of a random walk of length
$C \log n$, the probability that there are $1\leq r<l$, $0\leq s< l$ such that
\begin{equation}\label{eq:koloko}\begin{aligned}
1 + s &\equiv r^{\gamma} \mod l\\
a + s &\equiv (r+1)^{\gamma} \mod l
\end{aligned}\end{equation}
is positive and bounded from below ($\geq (l/n)^2 - o(1)$). Moreover, we
can check in time $O(n \log n)$ whether such $(r,s)$ exist. 
Taking a small number of random walks, it is almost certain
that we find a $\gamma$ and $(r,s)$ satisfying (\ref{eq:koloko}).
Then the short word
\[\gamma v^{-s} (1\; a\; b) v^s \gamma^{-1}\]
gives us a 3-cycle of
the form $(r\; r+1\; ?)$. Changing labels, we can write this as $(1\; 2\; x)$ for
some $x$.

Now $\phi := v (1\; 2\; x)^{-1}$
fixes $1$, and acts on $2,\dotsc,l$ as follows: 
$(2\; 3\; \dotsc\; x-1) (x\; x+1\; \dotsc\; l)$.
By conjugating $(1\; 2\; x)$ by a power of $\phi$, we can construct quickly
an element of the form $(1\; y\; ?)$ for any given $y$ in $\{2,\dotsc,l\}$. 
Conjugating
such an element by a power of $v$, we can construct quickly an element of
the form $(r\; s\; ?)$ for any $r, s \in \{1,2,...l\}$. Since
\[(r\; t\; ?)^{-1} (r\; s\; ?')^{-1} (r\; t\; ?) (r\; s\; ?') = (r\; s\; t)\]
for $?'$ distinct from $t$ and $?$ distinct from $s$, we see that we can construct
quickly any specified 3-cycle of the form 
$(r\; s\; t)$ with $1\leq r,s,t\leq l$.

In turn, this allows us to express any 3-cycle as a short enough word,
in the time we want: we simply try out random walks until we find one that
sends our $3$-cycle to a $3$-cycle with support
 within $\{1,2,\dotsc,l\}$, and then we apply the above
algorithm to that 3-cycle. (Since $l\gg n$, we succeed with probability
$1-O(1/n^C)$, $C$ arbitrary, after $\ll C \log n$ tries.)
In this way, we can construct all the 3-cycles we need - each in time 
$O\left(n (\log n)^{O(1)}\right)$.

Notice that we have tacitly assumed that we are representing (that is, 
store) our permutations either as we usually write them down 
(products of disjoint cycles) or
as maps from $\{1,2,\dotsc,n\}$ to itself; we can go from one of these
two forms of representation to the other one quickly.
 It may be more challenging to solve the 
problem without using heavily the particular way in which permutations
are stored -- or with constraints that forbid us to access directly
the internal representation of a permutation, whatever that representation
may be.


In the end, the importance of finding a solution in time $O(n (\log
n)^{O(1)})$ to the algorithmic problem we have discussed here will depend
on whether there are potential applications. 
There seem to be some: see \cite{MR2925926}.
\bibliographystyle{alpha}
\bibliography{hsz}
\end{document}